\documentclass[a4paper]{amsart}
\usepackage{amsfonts,amsmath,amssymb,amscd,amstext,amsbsy,latexsym}
\newtheorem{thm}{Theorem}[section]
\newtheorem{cor}[thm]{Corollary}
\newtheorem{lem}[thm]{Lemma}
\newtheorem{prop}[thm]{Proposition}
\theoremstyle{definition}

\theoremstyle{remark}
\newtheorem{rem}[thm]{Remark}
\newtheorem{exm}[thm]{Example}
\numberwithin{equation}{section}

\begin{document}

\title{The conductor ideals of Maximal subrings in Non-commutative Rings}%
\author{Alborz Azarang}%
\keywords{maximal subrings, conductor ideal, prime ideal, idealizer, integrally closed, finiteness condition}%
\subjclass[2010]{16D25; 16D80; 16N60; 16P20; 16P40}%

\maketitle

\centerline{Department of Mathematics, Faculty of Mathematical Sciences and Computer,}
\centerline{ Shahid Chamran University
of Ahvaz, Ahvaz-Iran} \centerline{a${}_{-}$azarang@scu.ac.ir}
\centerline{ORCID ID: orcid.org/0000-0001-9598-2411}
\begin{abstract}
Let $R$ be a maximal subring of a ring $T$, and $(R:T)$, $(R:T)_l$ and $(R:T)_r$ denote the greatest ideal, left ideal and right ideal of $T$ which is contained in $R$, respectively. It is shown that $(R:T)_l$ and $(R:T)_r$ are prime ideals of $R$ and $|Min_R((R:T))|\leq 2$. We prove that if $T_R$ has a maximal submodule, then $(R:T)_l$ is a right primitive ideal of $R$. We investigate that when $(R:T)_r$ is a completely prime (right) ideal of $R$ or $T$. We see that $Char(R/(R:T)_l)=Char(R/(R:T)_r)$, and if $Char(T)$ is neither zero or a prime number, then $(R:T)\neq 0$. If $|Min(R)|\geq 3$, then $(R:T)$ and $(R:T)_l(R:T)_r$ are nonzero ideals. If $R$ is integrally closed in $T$, then $(R:T)_l$ and $(R:T)_r$ are prime one-sided ideals of $T$, moreover $(R:T)$ is a semiprime ideal of $T$ and either $(R:T)$ is a prime ideal of $T$ or $(R:T)=(R:T)_l\cap (R:T)_r$ is a semiprime ideal of $R$. We observe that if $(R:T)_lT=T$, then $T$ is a finitely generated left $R$-module and $(R:T)_l$ is a finitely generated right $R$-module which is a right primitive ideal of $R$. Finally we study the Noetherian and the Artinian properties between $R$ and $T$.
\end{abstract}
\section{Introduction}
\subsection{Motivation}
In \cite{frd}, Ferrand and Oliver studied minimal ring extension of commutative rings. Note that whenever $R\subseteq T$ is a minimal extension of commutative ring (i.e., $R$ is a maximal subring of $T$), then the integral closure of $R$ in $T$, say $S$, is a subring between $R$ and $T$ and therefore by the minimality of the extension we infer that either $S=R$ (i.e., $R$ is integrally closed in $T$) or $S=T$ (i.e., $T$ is integral over $R$, equivalently $T$ is a finitely generated $R$-module). They proved that $T$ is integral over $R$ if and only if $(R:T)\in Max(R)$, see \cite[Proposition 4.1]{frd}. Moreover, if $R$ is integrally closed in $T$, then $(R:T)$ is a prime ideal of $T$, see \cite[Lemma 3.2]{frd}. Hence in any cases we deduce that $(R:T)$ is a prime ideal of $R$.\\

In this paper, motivated by the previous results, we are interested to study the conductor (left/right) ideals of minimal ring extension in non-commutative rings. More exactly, if $R\subseteq T$ is a minimal ring extension of rings, i.e., $R$ is a maximal subring of $T$ (where $T$ is an arbitrary ring which is not necessary commutative), then we would like to study the properties of $(R:T)$, the largest ideal of $T$ which is contained in $R$, $(R:T)_l$, the largest left ideal of $T$ which is contained in $R$, and $(R:T)_r$, the largest right ideal of $T$ which is contained in $R$. As we see in this paper these conductors are ideals of $R$ (and therefore we use "maximal subrings" instead of "minimal ring extension" in this paper). Moreover, in fact we proved that $(R:T)_l$ and $(R:T)_r$ are prime ideals of $R$.\\

It is interesting to know that each ring $R$ can be considered as a maximal subring of a larger ring $T$, see \cite[Theorem 3.7]{azq}. Note that if $T$ is a ring and $A$ be a one-sided ideal of $T$, then the idealizer of $A$ in $T$ is the largest subring of $T$, which $A$ is a two-sided ideal of it. More exactly, if $A$ is a right ideal of $T$, then the idealizer of $A$ is the subring $\mathbb{I}_T(A):=\{r\in T\ |\ rA\subseteq A\}$. By this definition we have the following from \cite{azq}:

\begin{thm}\label{pt1}
\begin{enumerate}
\item \cite[Theorem 4.1]{azq}. Let $T$ be a ring and $A$ be a maximal right/left ideal of $T$ which is not a two-sided ideal of $T$. Then the idealizer of $A$ is a maximal subring of $T$. In particular, either a ring has a maximal subring or is a quasi duo ring (i.e., each maximal left/right ideal is two-sided).
\item \cite[Proposition 4.2]{azq}. Let $R$ be a maximal subring of a ring $T$ which contains a maximal one-sided ideal $A$ of $T$ which is not a two-sided ideal of $T$. Then $R$ is the idealizer of $A$ in $T$.
\end{enumerate}
\end{thm}




In \cite{blgfc,blknfms,klein,laffey,lee}, the authors proved that if $R$ is a finite maximal subring of a ring $T$, then $T$ is finite too. In \cite[Theorem 3.8]{azkrc}, it is shown that if $R$ is a maximal subring of a commutative ring $T$, then $R$ is Artinian if and only if $T$ is Artinian and is integral over $R$, which immediately implies the latter facts for commutative rings. It is clear that, if $R$ is a maximal subring of a commutative ring $T$, then $T$ is Noetherian whenever $R$ is Noetherian, for $T=R[\alpha]$, for each $\alpha\in T\setminus R$ and use the Hilbert Basis Theorem. Motivated by these results in commutative case, we are interested to investigate the Artinian and the Noetherian properties in minimal ring extension of non-commutative rings too.\\

Finally, we refer the reader to \cite{dbsid,dbsc,glf,ghi,adjex,abmin}, for minimal ring extension of commutative rings and \cite{dorsy} for non-commutative case. Also, we refer the reader to \cite{modica,azkra,azkrc,azkrm} for maximal subrings in commutative rings and \cite{azq} for maximal subrings of non-commutative rings.

\subsection{Review of the results}
Assume that $R$ is a maximal subring of a ring $T$. In Section 2, we obtain some basic fact about the conductor ideals of $R$.  We prove that $(R:T)_l$ and $(R:T)_r$ are prime ideals of $R$. Moreover $Min_R((R:T))\subseteq \{(R:T)_l,(R:T)_r\}$ and in fact $|Min_R((R:T))|=1$ if and only if $(R:T)_l$ and $(R:T)_r$ are comparable if and only if $(R:T)$ is prime ideal of $R$. We show that if $T/R$ as right $R$-module has a maximal submodule, then $(R:T)_l$ is a right primitive ideal of $R$. Conversely, we observe that if either $(R:T)_l\in Max_r(R)$ or $R$ is a right Artinian ring, then $T/R$ has a maximal right $R$-submodule. We investigate whether $(R:T)_r$ is a completely prime (right) ideal of $T$ or of $R$. In particular, if $(R:T)_r$ is a completely prime right ideal of $T$, then $(R:T)_r$ is a completely prime ideal of $R$ (that is $R/(R:T)_r$ is a domain) and $T/(R:T)_r$ is a torsionfree left $R/(R:T)_r$-module; conversely if $T/(R:T)_r$ is a torsionfree left $R/(R:T)_l$-module, then either $(R:T)_r$ is an ideal of $T$ or $(R:T)_r$ is a completely prime right ideal of $T$ and $R$ is the idealizer of $(R:T)_r$ in $T$. Moreover, if $R$ is a right Artinian ring and $(R:T)_r$ is a completely prime right ideal of $T$, then either $R/P$ or $T/P$ is a division ring. We prove that the ring $End((T/R)_R)$ and therefore the ring $End_{\mathbb{Z}}(T/R)$ contain a copy of $R/(R:T)_r$. Consequently, the ring $End_{\mathbb{Z}}(T/R)$ contains a copies of $R/(R:T)_r$ and $R/(R:T)_l$. In particular, the characteristic of the ring $End_{\mathbb{Z}}(T/R)$ is either zero or a prime number and thus $Char(R/(R:T)_l)=Char(R/(R:T)_r)$. We show that if $Char(T)$ is not a prime number, then either there exists a prime number $p$ such that $pT\subseteq (R:T)$ (and therefore $(R:T)\neq 0$) or $Char(T)=Char(End_{\mathbb{Z}}(T/R))=0$. In Section 3, we show that if $R$ is ($2$-)integrally closed in $T$, then $(R:T)_l$ and $(R:T)_r$ are prime one-sided ideals of $T$. In particular, in this case $(R:T)$ is a semiprime ideal of $T$ and either $(R:T)$ is a prime ideal of $T$ or a semiprime ideal of $R$. Finally in Section 4, we consider some finiteness conditions on the extension $R\subseteq T$. We show that if $(R:T)_lT=T$, then $T$ is a finitely generated left $R$-module, $(R:T)_l$ is a finitely generated right ideal of $R$ and $R$ is the idealizer of $(R:T)_l$. In particular, if $(R:T)\in Max(T)\setminus Spec(R)$, then ${}_R T$, $T_R$, $(R:T)_l$ as right ideal of $R$ and $(R:T)_r$ as left ideal of $R$, are all finitely generated. Consequently, if $(R:T)\in Max(T)$ and $R$ is left Noetherian (resp. Artinian), then either $(R:T)=(R:T)_l$ or $T$ is a right Noetherian (resp. Artinian) ring. We show that if $R$ is a Noetherian (resp. an Artinian) ring and $(R:T)\neq 0$, then $T/l.ann_T((R:T)_l)$ and $T/r.ann_T((R:T)_r)$ are left and right Noetherian (resp. Artinian) rings, respectively. Moreover, if in addition $T$ is semiprime, then $T/ann_T((R:T))$ is a Noetherian (resp. an Artinian) ring. In particular, if $T$ is a prime ring, then $T$ is finitely generated as left and right $R$-modules and consequently $T$ is a Noetherian (resp. an Artinian) ring. Finally, if $R$ is a Noetherian ring, $T$ is not a prime ring and $R$ contains a prime ideal $Q$ of $T$, then either $T$ is Noetherian or $Q=(R:T)$ (therefore $Q$ is unique), $Q$ is a minimal prime ideal of $T$ and either $Q=(R:T)_l$ or $Q=(R:T)_r$.

\subsection{Notations and Definitions}
All rings in this paper are unital with $1\neq 0$ and subrings, modules and homomorphisms are also unital. If $R\subsetneq T$ is a ring extension and there exists no other subring between $R$ and $T$, then $R$ is called a maximal subring of $T$, or the extension $R\subseteq T$ is called a minimal ring extension. If $T$ is a ring, $R$ and $S$ are subrings of $T$, then clearly $T$ is a $(R,S)$-bimodule and therefore we can consider $(R,S)$-subbimodules of $T$. In particular, if $t\in T$, then the $(R,S)$-subbimodule of $T$ which is generated by $t$ is denoted by $RtS=\{\sum_{i=1}^n r_its_i\ |\ r_i\in R, s_i\in S, n\geq 0\}$. It is clear that if $R\subseteq S$ and $t\in S$, then $R+RtS$ and $R+StR$ are also subrings of $T$ which contain $R$. In particular, for each $t\in T$, the subrings $R+RtT$ and $R+TtR$ of $T$ contain $R$. More generally, if $I$ is left (resp. right) ideal of $T$, then $R+IR$ (resp. $R+RI$) is a subring of $T$ which contains $R$. If $T$ is a ring, $I$ is an ideal of $T$ and $M$ is a left (resp. right) $T$-module, then $Min_T(I)$, $Max_r(T)$, $Max_l(T)$, $Max(T)$, $l.ann_T(M)$ (resp. $r.ann_T(M)$) and $ann_T(I)$ denote the set of all minimal prime ideals of $I$ in $T$, the set of all maximal left ideals of $T$, the set of all maximal right ideals of $T$, the set of all maximal ideals of $T$, the left annihilator of $M$ in $T$ (resp. the right annihilator of $M$ in $T$) and the annihilator of $I$ in $T$ (i.e., $ann_T(I)=l.ann_T(I)\cap r.ann_T(I)$), respectively. We use $Min(T)$ for $Min_T(0)$. $J(T)$ denotes the Jacobson radical of a ring $T$ and for an ideal $I$ of $T$, we denote the prime radical of $I$ by $rad_T(I)$. The characteristic of a ring $T$ is denoted by $Char(T)$. If $X$ is a subset of a ring $T$, then $C_T(X)$ is the centralizer of $X$ in $T$, in particular, $C(T)=C_T(T)$ is the center of $T$. If $T$ is a ring, then $\mathbb{M}_n(T)$ denote the ring of all $n\times n$ square matrices over $T$. If $M$ is a left (resp. right) module over a ring $T$, then the ring of all $T$-module homomorphisms of $M$ is denoted by $End({}_TM)$ (resp. $End(M_T)$); when $T=\mathbb{Z}$ (i.e., $M$ is an abelian group), then we use $End_{\mathbb{Z}}(M)$.  A ring $T$ is called left (resp. right) quasi duo if each maximal left (resp. right) ideal of $T$ is a two-sided ideal of $T$, see \cite{lamq}. $T$ is called quasi duo if $T$ is left and right quasi duo ring. A ring $T$ is called left (resp. right) duo if each left (resp. right) ideal of $T$ is two-sided. Similarly duo rings are defined. If $R$ is a subring of a ring $T$ and $t\in T$, then we say that $t$ is left (resp. right) $n$-integral over $R$, if $t$ is a root of a left (resp. right) monic polynomial of degree $n$ over $R$ ($n\geq 1$). $R$ is called left (resp. right) $n$-integrally closed in $T$, if every left (resp. right) $n$-integral element of $T$ over $R$ belongs to $R$. $R$ is called $n$-integrally closed in $T$ whenever $R$ is left and right $n$-integrally closed in $T$. $R$ is called left (resp. right) integrally closed in $T$, if $R$ is left (resp. right) $n$-integrally closed in $T$, for each $n$. $R$ is integrally closed in $T$, if $R$ is left and right integrally closed in $T$. For other notations and definitions we refer the reader to \cite{good,lam,lam2,macrob,rvn}.

\section{The Conductor ideals of Maximal Subrings}
Let $R\subseteq T$ be a ring extension, then we have three types of conductor ideals for $R$: $(R:T):=\{x\in T\ |\ TxT\subseteq R\}$, $(R:T)_l:=\{x\in T\ |\ Tx\subseteq R\}$, $(R:T)_r:=\{x\in T\ |\ xT\subseteq R\}$. In other words, $(R:T)$ is the largest common ideal between $R$ and $T$, $(R:T)_l$ (resp. $(R:T)_r$)  is the largest common left (resp. right) ideal between $R$ and $T$. It is clear that $(R:T)_l=r.ann_R(\frac{T}{R})$ and therefore $(R:T)_l$ is a two sided ideal of $R$. Similarly, $(R:T)_r=l.ann_R(\frac{T}{R})$ is an ideal of $R$. Finally note that $(R:T)_l(R:T)_r\subseteq (R:T)\subseteq (R:T)_l\cap (R:T)_r$. In particular, if $R\neq T$, then $(R:T)_l+(R:T)_r\neq T$. Now we want to prove some generalization of the results in \cite{frd}. Let us first review some facts from \cite{frd}. If $T$ is a commutative ring and $R$ is a maximal subring of $R$, then $(R:T)$ is a prime ideal of $R$. In fact, since $R'$, the integral closure of $R$ in $T$, is a ring between $R$ and $T$, then by maximality of $R$ either $R'=T$, i.e., $T$ is integral over $R$ (equivalently $T$ is a finitely generated $R$-module) or $R'=R$, i.e., $R$ is integrally closed in $T$. Moreover, $T$ is integral over $R$ if and only if $(R:T)\in Max(R)$ (and therefore is a prime ideal of $R$); and if $R$ is integrally closed in $T$, then $(R:T)$ is a prime ideal of $T$ and therefore is a prime ideal of $R$ (moreover for each $x, y\in T$, if $xy\in R$, then $x\in R$ or $y\in R$). Hence in any cases, $(R:T)$ is a prime ideal of $R$. Now the following is in order.

\begin{lem}\label{t1}
Let $R$ be a maximal subring of a ring $T$. Then $(R:T)_l$ and $(R:T)_r$ are prime ideals of $R$.
\end{lem}
\begin{proof}
Let $a,b\in R$ and $aRb\subseteq (R:T)_l$. Thus $TaRb\subseteq R$. Now assume that $a\notin (R:T)_l$, i.e., $Ta\nsubseteq R$. Thus $TaR\nsubseteq R$. Since $R$ is a maximal subring of $T$ and $R+TaR$ is a subring of $T$ which properly contains $R$, we conclude that $R+TaR=T$. Thus $Tb=Rb+TaRb\subseteq R$, i.e., $b\in (R:T)_l$. Hence $(R:T)_l$ is a prime ideal of $R$. Similarly $(R:T)_r$ is a prime ideal of $R$.
\end{proof}

By the previous result we have some observation for conductor ideals of maximal subrings as follows.

\begin{prop}\label{t2}
Let $R$ be a maximal subring of a ring $T$. Then the following hold:
\begin{enumerate}
\item $rad_R((R:T))=(R:T)_l\cap (R:T)_r$.
\item $(R:T)$ is a prime ideal of $R$ if and only if either $(R:T)=(R:T)_l$ or $(R:T)=(R:T)_r$ if and only if $(R:T)_l$ and $(R:T)_r$ are comparable.
\item $|Min_R((R:T))|\leq 2$. In fact, $|Min_R((R:T))|=1$ if and only if $(R:T)_l$ and $(R:T)_r$ are comparable. Otherwise $Min_R((R:T))=\{(R:T)_l, (R:T)_r\}$ and either $(R:T)$ is a prime ideal of $T$ or there exist $Q_l$ and $Q_r$ in $Min_T((R:T))$ such that $Q_l\cap R=(R:T)_l$, $Q_r\cap R=(R:T)_r$, $R/(R:T)_l\cong T/Q_l$ and $R/(R:T)_r\cong T/Q_r$; 
\item $(R:T)$ is a semiprime ideal of $R$ if and only if $(R:T)=(R:T)_l\cap (R:T)_r$.
\item If $R$ is a zero dimensional ring, then $(R:T)_l$ and $(R:T)_r$ are maximal ideals of $R$. Hence either $(R:T)=(R:T)_l=(R:T)_r$ or $(R:T)_l+(R:T)_r=R$.
\item $(R:T)=r.ann_T(T/(R:T)_r)=l.ann_T(T/(R:T)_l)=l.ann_R(T/(R:T)_l)=r.ann_R(T/(R:T)_r)$. In particular, if $(R:T)_l\in Max_l(T)$ (resp. $(R:T)_r\in Max_r(T)$), then $(R:T)$ is a left (resp. right) primitive ideal of $T$.
\item If $(R:T)\in Max(R)$, then $(R:T)=(R:T)_l=(R:T)_r$.
\item If $J(R)$ is nilpotent (in particular, if $R$ is a one-sided Artinian ring), then $J(R)\subseteq (R:T)_l\cap (R:T)_r$ and $J(R)^2\subseteq (R:T)$.
\item If $Nil^*(R)$ is nilpotent (in particular, if $R$ is a one-sided Noetherian ring), then $Nil^*(R)\subseteq (R:T)_l\cap (R:T)_r$ and $(Nil^*(R))^2\subseteq (R:T)$.
\end{enumerate}
\end{prop}
\begin{proof}
Since $((R:T)_l\cap(R:T)_r)^2\subseteq (R:T)_l(R:T)_r\subseteq (R:T)\subseteq (R:T)_l\cap (R:T)_r$ and $(R:T)_l$, $(R:T)_r$ are prime ideals of $R$, we conclude that $rad_R((R:T))=(R:T)_l\cap (R:T)_r$. Hence $(1)$ holds. The first pard of $(2)$ is evident by $(1)$. It is clear that if either $(R:T)=(R:T)_l$ or $(R:T)=(R:T)_r$, then $(R:T)_l$ and $(R:T)_r$ are comparable. Conversely, suppose that $(R:T)_l$ and $(R:T)_r$ are comparable. For instance, let $(R:T)_l\subseteq (R:T)_r$. Let $x\in (R:T)_l$, therefore $Tx\subseteq (R:T)_l\subseteq (R:T)_r$. Thus $TxT\subseteq R$. Hence $x\in (R:T)$. Thus $(R:T)=(R:T)_l$. For $(3)$, by the proof of $(1)$ note that for each prime ideal $Q$ of $R$ which contains $(R:T)$, we have $(R:T)_l\subseteq Q$ or $(R:T)_r\subseteq Q$ and therefore $Min_R((R:T))\subseteq \{(R:T)_l, (R:T)_r\}$. Hence by $(1)$ we immediately conclude that $|Min_R((R:T))|=1$ if and only if $(R:T)_l$ and $(R:T)_r$ are comparable. For the next part of $(3)$, assume that $(R:T)_l$ and $(R:T)_r$ are incomparable and therefore $Min_R((R:T))=\{(R:T)_l, (R:T)_r\}$. Now note that if $A\subseteq B$ is a minimal ring extension with $(A:B)=0$, and $P$ is a minimal prime ideal of $A$, then either $B$ is a prime ring or $P$ is a contraction of a minimal prime ideal of $B$. To see this, note that $A\setminus P$, is a $m$-system in $A$ and therefore in $B$, thus there exists a prime ideal $Q$ of $B$ such that $Q\cap A\subseteq P$. If $Q=0$, then $B$ is prime, otherwise $A+Q=B$, for $A$ is a maximal subring of $B$ and $(A:B)=0$. Thus $A/(A\cap Q)\cong B/Q$ as rings. Hence $A\cap Q$ is a prime ideal of $A$ and therefore $Q\cap A=P$, for $P$ is minimal. Clearly we may assume that $Q$ is a minimal prime of $B$. Applying this fact to the minimal ring extension $R/(R:T)\subseteq T/(R:T)$, we deduce that either $(R:T)\in Spec(T)$ or there exist minimal prime ideals $Q_l$ and $Q_r$ of $Min_T((R:T))$ such that $Q_l\cap R=(R:T)_l$ and $Q_r\cap R=(R:T)_r$. If $Q_l\subseteq R$, then $(R:T)_l\subseteq Q_l\subseteq (R:T)\subseteq (R:T)_r$ which is absurd. Thus $Q_l$, and similarly $Q_r$, are not contained in $R$. Thus by maximality of $R$ we conclude that $R+Q_l=T=R+Q_r$ and therefore we have the ring isomorphisms $R/(R:T)_l\cong T/Q_l$ and $R/(R:T)_r\cong T/Q_r$. $(4)$ is evident by $(1)$. $(5)$ is clear. For $(6)$, let $x\in T$ (or $x\in R$), then $x\in (R:T)$ if and only if $TxT\subseteq R$ if and only if $xT\subseteq (R:T)_l$ if and only if $x\in l.ann_T(T/(R:T)_l)$ (or $x\in l.ann_R(T/(R:T)_l)$), hence $(R:T)=l.ann_T(T/(R:T)_l)=l.ann_R(T/(R:T)_l)$. The proof of the other equalities of $(6)$ are similar. The final part of $(6)$ is obvious, for whenever $(R:T)_l\in Max_l(T)$, then $T/(R:T)_l$ is a simple left $T$-module and therefore $l.ann_T(T/(R:T)_l)=(R:T)$ is a left primitive ideal of $T$ and hence a prime ideal of $T$. $(7)$ is clear, for $(R:T)_l$ and $(R:T)_r$ are proper ideals of $R$ which contains $(R:T)$. Finally for $(8)$ and $(9)$, let $I$ be a nilpotent one-sided ideal of $R$, then clearly $I\subseteq (R:T)_l\cap (R:T)_r$, for $(R:T)_l$ and $(R:T)_r$ are prime ideals of $R$. Consequently, $I^2\subseteq (R:T)_l(R:T)_r\subseteq (R:T)$.
\end{proof}

\begin{exm}
Let $T$ be a ring and $M$ be a maximal left ideal of $T$ which is not an ideal of $T$. If $R$ is the idealizer of $M$ in $T$, then by $(1)$ of Theorem \ref{pt1}, $R$ is a maximal subring of $T$ and clearly $(R:T)_l=M$. Moreover the following hold:
\begin{enumerate}
\item $T/R$ is a simple left $R$-module.
\item $R\subseteq T$ is a left integral extension.
\item $R/M$ is a division ring.
\item $(R:T)$ is a left primitive ideal of $T$.
\item $(R:T)_r$ is a left primitive ideal of $R$.
\end{enumerate}
For $(1)$ note that since $M$ is a maximal left ideal of $T$ which is not an ideal of $T$, for each $x\in T\setminus R$ we infer that $M+Mx=T$ and therefore $R+Rx=T$. This immediately shows that $T/R$ is a simple left $R$-module. Hence by $(R:T)_r=l.ann_R(T/R)$ we deduce $(5)$. Also note that $x^2\in T=R+Rx$ shows that $R\subseteq T$ is a left integral extension, hence $(2)$ holds. $(3)$ is obvious for $T/M$ is a simple left $T$-module and we have the rings isomorphism $R/M\cong End({}_T(T/M))$, by \cite[Lemma 1.3]{rob}. Finally, for $(4)$ note that by $(6)$ of Proposition \ref{t2}, $(R:T)=l.ann_T(T/M)$ and therefore $(R:T)$ is a left primitive ideal of $T$.
\end{exm}

It is not hard to see that in the previous example we also have $(R:T)\subsetneq (R:T)_l$ and therefore $(R:T)_l\nsubseteq (R:T)_r$. In the following example we give an example of a maximal subring $R$ of a ring $T$ with $(R:T)_l\neq (R:T)_r$ and $(R:T)\neq (R:T)_l\cap (R:T)_r$.

\begin{exm}\label{t3}
Let $D$ be a division ring and $T=\mathbb{M}_2(D)$. It is not hard to see that $R=\begin{pmatrix}
  D & 0 \\
  D & D
\end{pmatrix}$ is a maximal subring of $T$. It is clear that $(R:T)_l=\begin{pmatrix}
  D & 0 \\
  D & 0
\end{pmatrix}$, $(R:T)_r=\begin{pmatrix}
  0 & 0 \\
  D & D
\end{pmatrix}$, hence $(R:T)_l\cap(R:T)_r=\begin{pmatrix}
  0 & 0 \\
  D & 0
\end{pmatrix}\neq (R:T)=0$. Also note that $(R:T)=0$ is a prime (in fact maximal) ideal of $T$, but is not a prime ideal in $R$.
\end{exm}

Now we have the following result which is a generalization of \cite[Proposition 4.1]{frd}.

\begin{cor}\label{t5}
Let $R$ be a maximal subring of a ring $T$. If $T/R$ as right $R$-module has a maximal submodule, then $(R:T)_l$ is a right primitive ideal of $R$. In particular, if $T/R$ (i.e. $T_R$) is a finitely generated right $R$-module, then $(R:T)_l$ is a right primitive ideal of $R$.
\end{cor}
\begin{proof}
First we show that, if $M$ is a proper right $R$-submodule of $T$ which contains $R$, then $r.ann_R(T/M)=(R:T)_l$. Since $R\subseteq M$, we immediately conclude that $(R:T)_l\subseteq r.ann_R(T/M)$. Now assume that $x\in r.ann_R(T/M)$, but $x\notin (R:T)_l$. Thus $Tx\nsubseteq R$ and therefore $R+TxR=R$, for $R$ is a maximal subring of $T$. Since $Tx\subseteq M$ and $M$ is a right $R$-submodule of $T$, we deduce that $TxR\subseteq M$. Thus $T=R+TxR\subseteq M$, which is absurd. Thus $r.ann_R(T/M)\subseteq (R:T)_l$ and therefore the equality holds. Now by the assumption, let $M/R$ be a maximal right $R$-submodule of $T/R$, then $M$ is a maximal right $R$-submodule of $T$ and therefore $T/M$ is a simple right $R$-module. Therefore $r.ann_R(T/M)$ is a right primitive ideal of $R$. Hence by the first part of the proof we are done. The final part is evident.
\end{proof}

\begin{rem}\label{t6}
If $T$ is a commutative ring and $R$ is a maximal subring of $T$, then the following are equivalent:
\begin{enumerate}
\item $T$ is integral over $R$ (i.e., $T$ is a finitely generated $R$-module).
\item $(R:T)\in Max(R)$.
\item $T$ has a maximal $R$-submodule which contains $R$.
\item $T/R$ is a semisimple $R/(R:T)$-module.
\end{enumerate}
To see this, note that $(1)$ and $(2)$ are equivalent by \cite[Proposition 4.1]{frd}. Clearly, $(1)$ implies $(3)$, conversely $(3)$ implies $(2)$ by the previous corollary and the fact that in commutative ring maximal ideals and primitive ideals are coincided. Now assume that $(2)$ holds. Thus $T/R$ is a nonzero $R/(R:T)$-module. Since $R/(R:T)$ is a field, we immediately conclude that $T/R$ is a semisimple $R/(R:T)$-submodule and therefore $(4)$ holds. Finally, suppose that $T/R$ is a semisimple $R/(R:T)$-module, then $T/R$ has a maximal $R$-submodule and hence $(3)$ holds.
\end{rem}

In the next two results we see that the converse of Corollary \ref{t5}, holds (similar to the previous remark) in certain conditions.

\begin{cor}\label{t7}
Let $R$ be a maximal subring of a ring $T$. If $(R:T)_l\in Max_r(R)$, then $T$ has a maximal right $R$-submodule which contains $R$.
\end{cor}
\begin{proof}
Since $(R:T)_l\in Max_r(R)$, we conclude that $R/(R:T)_l$ is a division ring. Clearly, $T/R$ is a nonzero right $R/(R:T)_l$-module, and therefore has a maximal submodule.
\end{proof}

In \cite[Theorem 3.8]{azkrc}, it is proved that if $R$ is a maximal subring of a commutative ring $T$, then $R$ is Artinian if and only if $T$ is Artinian and is integral over $R$. An essential key for the proof of this result, by Remark \ref{t6}, is the fact that in this case $(R:T)$ is a maximal ideal of $R$. Now the following is in order.

\begin{prop}\label{t8}
Let $R$ be a right Artinian ring which is a maximal subring of a ring $T$. Then $(R:T)_l$ and $(R:T)_r$ are maximal ideals of $R$. $T/R$ as right $R/(R:T)_l$-modules (resp. $R/(R:T)_r$) is an isotypic semisimple (and therefore has a maximal submodule). In particular, if $R$ a right Artinian local ring which is a maximal subring of a ring $T$, then $(R:T)=(R:T)_l=(R:T)_r$.
\end{prop}
\begin{proof}
By Lemma \ref{t1}, $P:=(R:T)_l$ is a prime ideal of $R$ and since $R$ is a right Artinian ring we immediately conclude that $R/P\cong \mathbb{M}_n(D)$, for some division ring $D$ and a natural number $n$. Hence $R/P$ is a simple ring and therefore $P$ is a maximal ideal of $R$. Also note that $T/R$ is a nonzero right $R/P$-module. Thus $T/R$ is a semisimple $R/P$-module. Finally note that since $R/P\cong \mathbb{M}_n(D)$, we immediately conclude that each simple component of $T/R$ is isomorphic to other one. The proof for $(R:T)_r$ is similar. The final part is evident.
\end{proof}

Next, we want to discuss when the conductor ideals of a maximal subring $R$ of a ring $T$ are completely prime (right/left) ideals in $R$ or $T$. We remind the reader that
a proper right ideal $P$ of a ring $S$ is called completely prime right ideal if for each $a,b\in S$, whenever $aP\subseteq P$ and $ab\in P$, then $a\in P$ or $b\in P$, see \cite{reyes}. A prime ideal of a ring $S$ is called completely prime if $S/P$ is a domain.

\begin{prop}\label{t10}
Let $R$ be a maximal subring of a ring $T$ and $P=(R:T)_r$. If $P$ is a completely prime right ideal of $T$, then $R/P$ is a domain (i.e., $P$ is a completely prime ideal of $R$) and $T/P$ is a torsionfree left $R/P$-module. Conversely, if $P$ is a completely prime ideal of $R$ and $T/P$ is a torsionfree left $R/P$-module, then either $P$ is an ideal of $T$ (and therefore $(R:T)=P$) or $P$ is a completely prime right ideal of $T$ and $R=\mathbb{I}_T(P)$.
\end{prop}
\begin{proof}
First assume that $P$ is a completely prime right ideal of $T$. Let $a,b\in R$ and $ab\in P$. Since $P$ is an ideal of $R$, we conclude that $aP\subseteq P$. Therefore $a\in P$ or $b\in P$, for $P$ is a completely prime right ideal of $T$. Thus $R/P$ is a domain. Since $P=(R:T)_r$ is a right ideal of $T$ and $P$ is an ideal of $R$ (i.e., $PT\subseteq R$ and therefore $P(T/R)=0$), we conclude that $T/P$ is a left $R/P$-module. Now we show that $T/P$ is a torsionfree left $R/P$-module. Hence assume that $(r+P)(t+P)=0$, where $r\in R$ and $t\in T$. Thus $rt\in P$ and $rP\subseteq P$. Since $P$ is a completely prime right ideal of $T$, we infer that either $r\in P$ or $t\in P$ and therefore $T/P$ is a torsionfree left $R/P$-module. Conversely, suppose that $R/P$ is a domain and $T/P$ is a torsionfree left $R/P$-module. Let $a,b\in T$ such that $aP\subseteq P$ and $ab\in P$. We have two cases, either $P$ is an ideal of $T$ (and therefore $P=(R:T)$) or $P$ is not an ideal of $T$. In the latter case, since $P$ is an ideal of $R$ and $R$ is a maximal subring of $T$, we immediately conclude that $R=\mathbb{I}_T(P)$. Thus $a\in R$ and since $ab\in P$, we have $(a+P)(b+P)=0$. Therefore $a\in P$ or $b\in P$, for $T/P$ is a torsionfree left $R/P$-module. Thus $P$ is a completely prime right ideal of $T$.
\end{proof}

Now the following is in order.

\begin{cor}\label{t11}
Let $R$ be a maximal subring of a ring $T$ and $P=(R:T)_r\in Max_r(T)$. Then $P$ is a completely prime ideal of $R$ (i.e., $R/P$ is a domain) and $T/P$ is a torsionfree left $R/P$-module.
\end{cor}
\begin{proof}
Note that by \cite[$(A)$ of Corollary 2.10]{reyes}, $P$ is a completely prime right ideal of $T$ and therefore we are done by the previous proposition.
\end{proof}

If we add the assumption that $T$ is a right Artinian ring in the first part of Proposition \ref{t10}, we obtain the following better conclusion.

\begin{cor}\label{t12}
Let $R$ be a maximal subring of a right Artinian ring $T$ and $P=(R:T)_r$ be a completely prime right ideal of $T$. The either $T/P$ or $R/P$ is a division ring.
\end{cor}
\begin{proof}
We have two cases, either $P$ is an ideal of $T$ or not. First assume that $P$ is an ideal of $T$, then by \cite[$(B)$ of Corollary 2.10]{reyes}, we conclude that $P\in Max_r(T)$ and therefore $T/P$ is a division ring. If $P$ is not an ideal of $T$, then $R=\mathbb{I}_T(P)$, for $R$ is a maximal subring of $T$ and $P$ is an ideal of $R$. Again by \cite[$(B)$ of Corollary 2.10]{reyes}, $P\in Max_r(R)$ and hence $R/P$ is a division ring.
\end{proof}

When $T$ is a (left/right) duo ring and $R$ is a maximal subring of $R$, then $(R:T)$ is a completely prime ideal of $R$, by our next fact.

\begin{prop}\label{t28}
Let $R$ be a maximal subring of a left duo ring $T$. Then $(R:T)=(R:T)_l$ is a completely prime ideal of $R$. Moreover, if $T$ is a duo ring, then $(R:T)=(R:T)_l=(R:T)_r$ is a completely prime ideal of $R$.
\end{prop}
\begin{proof}
Since $T$ is a left duo ring, we immediately conclude that $(R:T)=(R:T)_l$. Now assume that $a,b\in R$ such that $ab\in (R:T)$ and $a\notin (R:T)=(R:T)_l$. Therefore $Tab\subseteq R$ and $Ta\nsubseteq R$. Since $R$ is a maximal subring of $T$ and $Ta$ is an ideal of $T$, we infer that $R+Ta=T$. Hence $Tb=(R+Ta)b=Rb+Tab\subseteq R$, i.e., $b\in (R:T)=(R:T)_l$. Thus $(R:T)=(R:T)_l$ is a completely prime ideal of $R$. The final part is evident.
\end{proof}

In the next proposition we have a relation between the centralizer of a maximal subring $R$ of a ring $T$ and the conductor ideals.

\begin{prop}\label{t19}
Let $R$ be a maximal subring of a ring $T$. The one of the following holds:
\begin{enumerate}
\item $C_T(R)\subseteq R$. Therefore $R$ contains the center of $T$.
\item $(R:T)=(R:T)_l=(R:T)_r$, and there exists $\alpha\in T$ such that $T=R[\alpha]$, where $\alpha r=r\alpha$, for each $r\in R$.
\end{enumerate}
\end{prop}
\begin{proof}
Assume that $R$ does not contain $C_T(R)$. If $R=\mathbb{I}_T((R:T)_r)$, then clearly $C_T(R)\subseteq C_T((R:T)_r)\subseteq \mathbb{I}_T((R:T)_r)=R$, which is absurd. Thus $(R:T)_r$ is an ideal of $T$. Similarly, $(R:T)_l$ is an ideal of $T$ and hence the first part of $(2)$ holds. The second part is evident for $C_T(R)\nsubseteq R$.
\end{proof}

As we see in Proposition \ref{t10}, and in the proofs of Corollary \ref{t12} and the previous proposition, a maximal subring $R$ of a ring $T$ is of the form an idealizer of a one-sided ideal of $T$. In the next remark we exactly determine when $R$ is of the form an idealizer of $T$.

\begin{rem}\label{t13}
Let $R$ be a maximal subring of a ring $T$. Then exactly  one of the following holds:
\begin{enumerate}
\item $(R:T)=(R:T)_l\subsetneq (R:T)_r$ and $R=\mathbb{I}_T((R:T)_r)$.
\item $(R:T)=(R:T)_r\subsetneq (R:T)_l$ and $R=\mathbb{I}_T((R:T)_l)$.
\item $(R:T)\subsetneq (R:T)_r, (R:T)_l$ and $R=\mathbb{I}_T((R:T)_r)=\mathbb{I}_T((R:T)_l)$.
\item $(R:T)=(R:T)_l=(R:T)_r$.
\end{enumerate}
Moreover, if $(R:T)$ is not a prime ideal of $R$, then $(3)$ holds. To see this, note that $(R:T)\subseteq (R:T)_l\cap (R:T)_r$. Hence if either $(R:T)_l$ or $(R:T)_r$ is an ideal of $T$, then we immediately conclude that $(R:T)=(R:T)_l$ or $(R:T)=(R:T)_r$. For instance, if $(R:T)_l$ is an ideal of $T$, but $(R:T)_r$ is not an ideal of $T$, then $(R:T)=(R:T)_l$ and since $(R:T)_r$ is an ideal of $R$, we conclude that $R=\mathbb{I}_T((R:T)_r)$, for $R$ is a maximal subring of $T$ which is contained in the proper subring $\mathbb{I}_T((R:T)_r)$ of $T$. Thus $(1)$ holds. By a similar argument we see $(2)$. If $(R:T)_l$ and $(R:T)_r$ are ideals of $T$, then clearly $(4)$ holds. Finally, assume that $(R:T)_l$ and $(R:T)_r$ are not ideals of $T$. Since these are ideals of $R$ and $R$ is a maximal subring of $T$, we deduce that $R=\mathbb{I}_T((R:T)_l)=\mathbb{I}_T((R:T)_r)$. The final part is evident by Lemma \ref{t1}.
\end{rem}

\begin{rem}\label{t14}
Let $T$ be a ring and $A$ be a right ideal of $T$, then the map $\phi:\mathbb{I}(A)/A\longrightarrow End((T/A)_T)$, where $\phi(r+A)=f_{r+A}$ and $f_{r+A}(x+T)=rx+A$, for each $r\in\mathbb{I}(A)$ and $x\in T$ is a ring isomorphism, see \cite[Lemma 1.3]{rob}. In particular, if $R$ is a maximal subring of $T$ and $A=(R:T)_r$, then either $R/A\cong End((T/A)_T)$ (and therefore $End((T/A)_T)$ is a prime ring) or $\phi(R/A)$ is a maximal subring of $End((T/A)_T)\cong T/A$.
\end{rem}

As we see in the previous remarks, there are a closed relation between the idealizer of a one-sided ideal $A$ of a ring $T$ and the endomorphism ring of $T/A$. In our next results in this section we focus on endomorphism rings to find some facts for maximal subrings. As we see below, there exist some algebraic properties between the quotient rings of a maximal subring $R$ of a ring $T$ respect to conductor ideals and the endomorphism rings of $T/R$ as left/right $R$-modules and as $\mathbb{Z}$-module. First we need the following lemma similar to previous remark.

\begin{lem}\label{t15}
Let $R$ be a subring of a ring $T$ and $A=(R:T)_r$. Then the map $\psi: R/A\longrightarrow End((T/R)_R)$, where $\psi(r+A)=g_{r+A}$ and $g_{r+A}(x+R)=rx+R$, is a ring monomorphism, i.e., up to ring isomorphism, $R/A$ is a subring of  $End((T/R)_R)$. In particular, up to ring isomorphism $R/A$ is a subring of $End_{\mathbb{Z}}(T/R)$.
\end{lem}
\begin{proof}
The proof of $\psi$ is a ring homomorphism is similar to \cite[Lemma 1.3]{rob}. Only we prove that $Ker(\psi)=0$. To see this note that $r+A\in Ker(\psi)\Longleftrightarrow g_{r+A}=0\Longleftrightarrow g_{r+A}(t+R)=0$ for each $t\in T$, $\Longleftrightarrow rt+R=0$, for each $t\in T$, i.e., $rT+R=0\Longleftrightarrow rT\subseteq R\Longleftrightarrow r\in A$.
\end{proof}

We remind that since we always write maps on the left, then the Remark \ref{t14}, for left ideals and Lemma \ref{t15}, for $(R:T)_l$, have some different statements. Hence assume that $T$ be a ring and $B$ be a left ideal of $T$, then $\mathbb{I}(B)/B\cong (End({}_T(T/B)))^{op}$, where the notation $op$ means the opposite of a ring. Also, if $R$ is a subring of $T$ and $B=(R:T)_l$, then the ring $R/B$ embeds in $(End({}_R(T/R)))^{op}$ and therefore $(R/B)^{op}$ embeds in the ring $End({}_R(T/R))$ and hence in the ring $End_{\mathbb{Z}}(T/R)$. It is clear that if $S$ is any ring, then $Char(S)=Char(S^{op})$ and $S$ is a prime ring if and only if $S^{op}$ is a prime ring. Now the following is in order.

\begin{prop}\label{t16}
Let $R$ be a maximal subring of a ring $T$. Then $Char(End_{\mathbb{Z}}(T/R))$ is either $0$ or is a prime number. Moreover, $Char(R/(R:T)_l)=Char(R/(R:T)_r)=Char(R/(R:T))$.
\end{prop}
\begin{proof}
By Lemma \ref{t1}, $(R/(R:T)_l)^{op}$ and $R/(R:T)_r$ are prime rings, which are subrings of $End_{\mathbb{Z}}(T/R)$, by the previous lemma and its comment. Hence $Char(R/(R:T)_l)=Char(End_{\mathbb{Z}}(T/R))=Char(R/(R:T)_r)$. Now by the facts that the center of a prime ring is an integral domain and the characteristic of an integral domain is either $0$ or is a prime number, we conclude the first part of the proposition. For the final part, we have two cases: If $Char(R/(R:T)_l)=Char(R/(R:T)_r)=0$, then it is clear that $Char(R/(R:T))=0$, for $R/(R:T)_l$ and $R/(R:T)_r$ are quotients of the ring $R/(R:T)$. Hence assume that $Char(R/(R:T)_l)=Char(R/(R:T)_r)=p$, where $p$ is a prime number. Thus $pT=Tp\subseteq R$ and clearly $pT=Tp=TpT\subseteq R$ and thus $Char(R/(R:T))=p$.
\end{proof}

Now we have the following immediate corollary.

\begin{cor}
Let $R$ be a maximal subring of a ring $T$. Assume that the map $Char: Spec(R)\longrightarrow \mathbb{N}\cup\{0\}$, where $P\longmapsto Char(R/P)$ is one-one. Then $(R:T)_l=(R:T)_r$, in particular $(R:T)=(R:T)_l=(R:T)_r$ is a prime ideal of $R$ which is an ideal of $T$.
\end{cor}

In the following, we give some examples of rings with the property that assumed in the previous corollary.

\begin{exm}
\begin{enumerate}
\item Assume that $R:=\mathbb{M}_n(\mathbb{Z})$, where $n>1$ is a natural number. Then for each prime number $q$, it is clear that $\mathbb{Z}$ is a maximal subring of $S:=\mathbb{Z}[1/q]$ and therefore $R$ is a maximal subring of $T:=\mathbb{M}_n(S)$. It is obvious that, the map $Char$ mentioned in the previous corollary is one-one for $R$. Hence $(R:T)=(R:T)_l=(R:T)_r$.
\item Let $R=R_1\times\cdots\times R_n$, where $n>1$ and each $R_i$ is a simple ring with $Char(R_i)=p_i$ is a prime number. Assume that $p_i\neq p_j$ for $i\neq j$. If $R$ is a maximal subring of a ring $T$, then $(R:T)=(R:T)_l=(R:T)_r=R_1\times\cdots\times R_{i-1}\times 0\times R_{i+1}\times\cdots\times R_n$, for some $i$.
\item Let $R$ be a one-sided Artinian ring which is a maximal subring of a ring $T$. Assume that $R/J(R)\cong \mathbb{M}_{n_1}(D_1)\times\ldots\mathbb{M}_{n_k}(D_k)$, where $n_i$ and $k$ are natural numbers and $D_i$ is a division ring for each $i$. Then either $(R:T)=(R:T)_l=(R:T)_r$ or there exist $i\neq j$ such that $Char(D_i)=Char(D_j)$.
\end{enumerate}
\end{exm}

In the next two corollary, we obtain more facts about relations between the characteristic, conductor ideals and related endomorphism ring.

\begin{cor}\label{t17}
Let $R$ be a maximal subring of a ring $T$ with $Char(T)$ is not a prime number. Then either there exists a prime number $p$ such that $pT\subseteq (R:T)$ or $Char(T)=Char(End_{\mathbb{Z}}(T/R))=0$.
\end{cor}
\begin{proof}
By Proposition \ref{t16}, $Char(End_{\mathbb{Z}}(T/R))$ is either $0$ or is a prime number. If $Char(End_{\mathbb{Z}}(T/R))=p$, where $p$ is a prime number, then $p.1_{T/R}=0$ (note that $1_{T/R}$ denotes the identity map of $T/R$, i.e., $1_{T/R}(t+R)=t+R$, for each $t\in T$). Hence for each $t\in T$, $pt+R=0$, i.e., $pT\subseteq R$. Thus $0\neq pT\subseteq (R:T)$, (note $Char(T)$ is note a prime number, therefore $pT\neq 0$ and it is clear that $pT=Tp$ is an ideal of $T$). Otherwise, assume that $Char(End_{\mathbb{Z}}(T/R))=0$. If $Char(T)=m>0$, then for each $\phi\in End_{\mathbb{Z}}(T/R)$, we have $m\phi=0$, for $m\phi(t)=\phi(mt)=\phi(0)=0$, for any $t\in T$. Thus $Char(End_{\mathbb{Z}}(T/R))>0$ which is absurd. Therefore $Char(T)=Char(End_{\mathbb{Z}}(T/R))=0$ and hence we are done.
\end{proof}

The following is immediate now.

\begin{cor}\label{t18}
Let $R$ be a maximal subring of a ring $T$. Assume that $Char(T)$ is neither zero or a prime number. Then $(R:T)\neq 0$.
\end{cor}

finally in this section we have some observations about the annihilators of conductor ideals of maximal subrings.

\begin{prop}\label{t25}
Let $R$ be a maximal subring of a ring $T$. If $l.ann_T((R:T))+r.ann_T((R:T))=T$, then $(R:T)^2=0$.
\end{prop}
\begin{proof}
By the assumption, there exist $a\in l.ann_T((R:T))$ and $b\in r.ann_T((R:T))$ such that $a+b=1$. Therefore $(R:T)a=(R:T)$ and $(R:T)=b(R:T)$. Now note that $ab\in l.ann_T((R:T))\cap r.ann_T((R:T))$. Thus $(R:T)^2=(R:T)(R:T)=(R:T)ab(R:T)=0$.
\end{proof}

In particular, if $R$ is a maximal subring of a semiprime ring $T$, and the characteristic of $T$ is neither zero or a prime number, then $l.ann_T((R:T))+r.ann_T((R:T))$ is a proper ideal of $T$, by Corollary \ref{t18} and the previous result. As we see in the next conclusion, in certain rings a left (resp. right) conductor ideal of a maximal subring is the annihilator of the right (resp. left) conductor ideal.

\begin{thm}\label{t26}
Let $R$ be a maximal subring of a ring $T$. Then the following hold:
\begin{enumerate}
\item If $(R:T)_l$ and $(R:T)_r$ are incomparable, then $l.ann_R((R:T)_l)+r.ann_R((R:T)_l)\subseteq (R:T)_r$ and $l.ann_R((R:T)_r)+r.ann_R((R:T)_l)\subseteq (R:T)_l$.
\item If $(R:T)_l(R:T)_r=0$ (in particular, if $(R:T)=0$), then $Min(R)\subseteq \{(R:T)_l,(R:T)_r\}$. In particular, if $(R:T)_l$ and $(R:T)_r$ are incomparable, then $(R:T)_l=l.ann_R((R:T)_r)$ and $(R:T)_r=r.ann_R((R:T)_l)$ are exactly minimal prime ideals of $R$.
\item If $(R:T)_l\cap (R:T)_r=0$ and $(R:T)_l\neq 0\neq (R:T)_r$, then $(R:T)_l=l.ann_R((R:T)_r)=r.ann_R((R:T)_r)$ and $(R:T)_r=r.ann_R((R:T)_l)=l.ann_R((R:T)_l)$ are exactly minimal prime ideals of $R$. In particular, $T$ is not a prime ring.
\item If $R$ is a reduced ring, $(R:T)_l\cap (R:T)_r=0$ and $(R:T)_l\neq 0\neq (R:T)_r$, then $(R:T)_l$ and $(R:T)_r$ are completely prime ideals of $R$. In particular, $R$ embeds in a product of two domains.
\end{enumerate}
\end{thm}
\begin{proof}
$(1)$ is clear, for $l.ann_R((R:T)_l)(R:T)_l=(R:T)_lr.ann_R((R:T)_l)=0\subseteq (R:T)_r$ and $(R:T)_r$ is a prime ideal of $R$, which does not contain $(R:T)_l$ by our assumption. For $(2)$, first note that $(R:T)_l(R:T)_r\subseteq (R:T)$, hence if $(R:T)=0$, we conclude that $(R:T)_l(R:T)_r=0$. Hence assume that $(R:T)_l(R:T)_r=0$. This immediately shows that $Min(R)\subseteq \{(R:T)_l, (R:T)_r\}$, by Lemma \ref{t1}. Now assume that $(R:T)_l$ and $(R:T)_r$ are incomparable. Let $P=(R:T)_l$ and $I=(R:T)_r$. Thus $PI=0$ and $I\neq 0$ (by incomparability). It is clear that $P\subseteq l.ann_R(I)$. Now since $l.ann_R(I)I=0\subseteq P$ and $P$ is a prime ideal of $R$ which does not contain $I$, we deduce that $P=l.ann_R(I)$. Finally, if $P$ is not a minimal ideal of $R$, then let $Q$ be a prime ideal which properly is contained in $P$. Then $l.ann_R(I)I=0\subseteq Q\subsetneq P=l.ann_R(I)$ implies that $I\subseteq Q\subseteq P$, which is absurd for $I$ and $P$ are incomparable. Hence $P$ is a minimal prime ideal of $R$. Similarly, $(R:T)_r=r.ann_R((R:T)_l)$ is a minimal prime ideal of $R$ and hence $(2)$ holds. For $(3)$, note that since $(R:T)_l$ and $(R:T)_r$ are ideals of $R$, then we infer that $(R:T)_l(R:T)_r$ and $(R:T)_r(R:T)_l$ are subsets of $(R:T)_l\cap (R:T)_r$. Hence if $(R:T)_l\cap (R:T)_r=0$, we conclude that $(R:T)_l(R:T)_r=0$ and $(R:T)_r(R:T)_l=0$. By a similar argument of $(2)$, the conclusion of $(3)$ can be proved. For the final part of $(3)$, since $0\neq (R:T)_l=r.ann_R((R:T)_r)\subseteq r.ann_T((R:T)_r)$, we immediately deduce that $T$ is not a prime ring. For $(4)$, since $R$ is reduced, we conclude that $(R:T)_l$ and $(R:T)_r$ are incomparable and therefore by $(2)$ or $(3)$, are minimal prime ideal of $R$. Hence by \cite[Lemma 12.6]{lam}, $(R:T)_l$ and $(R:T)_r$ are completely prime ideals of $R$, i.e., $R/(R:T)_l$ and $R/(R:T)_r$ are domains and clearly $R$ embeds in product of them.
\end{proof}

By $(2)$ of the previous theorem we have the following immediate result.

\begin{cor}\label{t27}
Let $R$ be a ring which is a maximal subring of a ring $T$ with $|Min(R)|\geq 3$, then $(R:T)$ and $(R:T)_l(R:T)_r$ are nonzero.
\end{cor}

\section{Integrally closed maximal subrings}

As we mentioned in the introduction of this paper, if $R$ is a maximal subring of a commutative ring $T$ and $R$ is integrally closed in $T$, then $(R:T)$ is a prime ideal in $T$. Now we want to study when $(R:T)_r$ and $(R:T)_l$ are  prime one-sided ideals in non-commutative case. First we begin by the following main result.

\begin{thm}\label{t29}
Let $R$ be a maximal subring of a ring $T$, $P=(R:T)_r$, $a,b\in T$ such that $aTb\subseteq P$. If $a,b\notin P$, then the following hold:
\begin{enumerate}
\item $a$ is right $(2-)$integral over $R$ and $b$ is left $(2-)$integral over $R$.
\item If $a\notin R$, then $T=R+RaR$.
\item If $b\notin R$, then $T=R+RbR$.
\end{enumerate}
In particular, if $b\notin R$, then $a\notin R$; and if $a\notin R$, then either $b\notin R$ or $b\in (R:T)_l$.
\end{thm}
\begin{proof}
First note that by our assumption, $aTbT\subseteq R$, $aT\nsubseteq R$ and $bT\nsubseteq R$. Hence by maximality of $R$ we deduce that $R+RaT=T=R+RbT$. Thus there exist $r_i\in R$ and $s_i\in T$ such that $a=r_0+r_1bs_1+\cdots+r_nbs_n$. By multiplying this equation from left to $a$, we have $a^2=ar_0+ar_1bs_1+\cdots+ar_nbs_n$. Now for each $i$, $ar_ibs_i\in aTbT\subseteq R$. Therefore $r:=ar_1bs_1+\cdots+ar_nbs_n\in R$ and $a^2=ar_0+r$ which shows $a$ is right $(2-)$integral over $R$. By a similar argument (from $aTb\subseteq P\subseteq R$ and $T=R+RaT$) we deduce that $b$ is left $(2-)$integral over $R$. Hence $(1)$ holds. Now we prove $(2)$. Assume that $a\notin R$. We claim that $S:=R+RaR$ is a subring of $T$. To see this it suffices to show that if $x,y,u,v\in R$, then $(xay)(uav)\in S$. We have $(xay)(uav)=xayu(r_0+r_1bs_1+\cdots+r_nbs_n)v=xayur_0v+xayur_1bs_1v+\cdots+xayur_nbs_nv$. The first term of the summation is clear in $RaR$ and therefore in $S$. For the other terms of summation note that since $ayur_ibs_i\in aTbT\subseteq R$, we immediately conclude that $xayur_ibs_iv\in R$. Therefore $xayur_1bs_1v+\cdots+xayur_nbs_nv\in R$. Hence $(xay)(uav)\in S$ and thus $S$ is a subring of $T$, which properly contains $R$ (for $a\notin R$). Thus $S=T$, for $R$ is a maximal subring of $T$. The proof of $(3)$ is similar (by the use of $aTb\subseteq P\subseteq R$). Finally for the final part of the theorem, assume that $b\notin R$ but $a\in R$. Thus $T=R+RbR$ by $(3)$. Therefore $aT=aR+aRbR\subseteq R+aTbT\subseteq R$, i.e., $a\in P$ which is absurd. Now assume that $a\notin R$ but $b\in R$. Thus by $(2)$, $T=R+RaR$ and therefore $Tb=Rb+RaRb\subseteq R$. Thus $b\in (R:T)_l$.
\end{proof}

\begin{rem}\label{t30}
Similar to the previous result and its proof one can prove that, if $R$ is a maximal subring of a ring $T$, $P=(R:T)_l$, $a,b\in T$ such that $aTb\subseteq P$ and $a,b\notin P$, then the following hold:
\begin{enumerate}
\item $a$ is right $(2-)$integral over $R$ and $b$ is left $(2-)$integral over $R$.
\item If $a\notin R$, then $T=R+RaR$.
\item If $b\notin R$, then $T=R+RbR$.
\end{enumerate}
In particular, if $a\notin R$, then $b\notin R$; and if $b\notin R$, then either $a\notin R$ or $a\in (R:T)_l$.
\end{rem}

Let $T$ be a ring and $P$ be a proper one-sided ideal of $T$. $P$ is called prime if for each $a,b\in T$, from $aTb\subseteq P$ we deduce that either $a\in P$ or $b\in P$. For example, each maximal one-sided ideal $M$ of a ring $T$ is prime. To see this, assume that $M$ be a maximal left ideal of $T$ and $a,b\in T$ such that $aTb\subseteq M$. If $b\in M$, then we are done. Hence assume that $b\notin M$, therefore $M+Tb=T$, by maximality of $M$. Hence $aT=aM+aTb\subseteq M$ and therefore $M$ is prime. Now the following is in order.

\begin{cor}\label{t31}
Let $R$ be a maximal subring of a ring $T$ which is $(2-)$integrally closed in $T$, then $(R:T)_l$ and $(R:T)_r$ are prime one-sided ideals of $T$.
\end{cor}

For the ideal $(R:T)$ of an integrally closed maximal subring $R$ of a ring $T$, we have the following remark.

\begin{rem}\label{t32}
Let $R$ be a maximal subring of a ring $T$ which is $(2-)$integrally closed in $T$. If $a,b\in T$ such that $aTb\subseteq (R:T)$, but $a,b\notin (R:T)$, then the following hold:
\begin{enumerate}
\item $a,b\in R$.
\item $RaT\subseteq R$, but $TaR\nsubseteq R$. In particular, $T=R+TaR$.
\item $TbR\subseteq R$, but $RbT\nsubseteq R$. In particular, $T=R+RbT$.
\item $a\in (R:T)_r$, $b\in (R:T)_l$, $a\notin (R:T)_l$ and $b\notin (R:T)_r$.
\end{enumerate}
To see these, for $(1)$ we show that $a\in R$, the proof for $b\in R$ is similar. Since $a,b\notin (R:T)$, we infer that $TaT$ and $TbT$ are not contained in $R$. Therefore by maximality of $R$, we deduce that $R+TaT=T$ and $R+TbT=T$. Hence $a\in R+TbT$. Therefore there exist $n\in\mathbb{N}$, $x_i,y_i\in T$ for $1\leq i\leq n$ and $r\in R$ such that
$a=r+x_1by_1+\cdots+x_nby_n$. Multiplying by $a$ from left we obtain that $a^2=ar+ax_1by_1+\cdots+ax_nby_n$. Now note that for each $i$, $ax_ib\in aTb\subseteq (R:T)$ and $(R:T)$ is an ideal of $T$ which is contained in $R$, therefore $ax_iby_i\in (R:T)\subseteq R$ for each $i$. Hence $s:=ax_1by_1+\cdots+ax_nby_n\in R$ and $a^2=ar+s$ which means $a$ is right $(2-)$integral over $R$ and therefore $a\in R$. Thus $(1)$ holds. $(4)$ is obvious from $(2)$ and $(3)$. We prove $TaR\nsubseteq R$ and $TbR\subseteq R$, the proofs of the other parts of $(3)$ and $(4)$ are similar. Since $b\notin (R:T)$ we conclude that $TbT\nsubseteq R$ and therefore by maximality of $R$, we deduce that $R+TbT=T$. Hence $aR+aTbT=aT$ and therefore $TaR+TaTbT=TaT$. From $a\notin (R:T)$, we have $TaT\nsubseteq R$ and form $aTb\subseteq (R:T)$ we have $TaTbT\subseteq R$. Thus by $TaR+TaTbT=TaT$, we conclude that $TaR\nsubseteq R$ and therefore $Ta\nsubseteq R$. Hence $a\notin (R:T)_l$ and since $aTb\subseteq (R:T)\subseteq (R:T)_l$, we infer that $b\in (R:T)_l$, by the previous corollary. Hence $Tb\subseteq R$ and therefore $TbR\subseteq R$.
\end{rem}

Finally in this section we want to prove that whenever $R$ is an integrally closed maximal subring of a ring $T$, then $(R:T)$ is a semiprime ideal of $T$. First we need the next lemma.

\begin{lem}\label{t33}
Let $R$ be a maximal subring of a ring $T$ and $x\in (R:T)_l\cup (R:T)_r$. If $xTx\subseteq (R:T)$, then $x\in (R:T)$.
\end{lem}
\begin{proof}
Assume that $x\in (R:T)_l$, thus $Tx\subseteq R$ and therefore $TxR\subseteq R$. Suppose $x\notin (R:T)$, thus $TxT\nsubseteq R$ and hence $R+TxT=T$, by maximality of $R$. Since $xTx\subseteq (R:T)$, we infer that $TxTxT\subseteq R$. Therefore $TxT=TxR+TxTxT\subseteq R$, which is absurd. Thus $x\in (R:T)$. The proof for the case $x\in (R:T)_r$ is similar.
\end{proof}

Now the following is in order.

\begin{thm}\label{t34}
Let $R$ be a maximal subring of a ring $T$ which is (2-)integrally closed in $T$ (or whenever $x\in T$ and $x^2\in R$, then $x\in R$). Then $(R:T)$ is a semiprime ideal of $T$. Moreover, either $(R:T)$ is a prime ideal of $T$ or $(R:T)=(R:T)_l\cap (R:T)_r$ (is a semiprime ideal of $R$).
\end{thm}
\begin{proof}
We must prove that when $x\in T$ and $xTx\subseteq (R:T)$, then $x\in (R:T)$. Since $xTx\subseteq R$, we infer that $x^2\in R$ and therefore $x\in R$, by our assumption. Hence $xRx\subseteq xTx\subseteq (R:T)\subseteq (R:T)_l\cap (R:T)_r$. Therefore $xRx\subseteq (R:T)_l$ and hence $x\in (R:T)_l$, for $(R:T)_l$ is a prime ideal of $R$, by Lemma \ref{t1}. Thus $x\in (R:T)$, by the previous lemma. Hence $(R:T)$ is a semiprime ideal of $T$. Therefore there exists a family $Q_i$, $i\in I$, of prime ideals of $T$, such that $(R:T)=\bigcap_{i\in I} Q_i$. Now we have two cases. If there exists $i\in I$, such that $Q_i\subseteq R$, then $Q_i\subseteq (R:T)$, and therefore $(R:T)=Q_i$ is a prime ideal of $T$. Hence assume that for each $i\in I$, $Q_i\nsubseteq R$. Thus for each $i\in I$, we conclude that $R+Q_i=T$, by maximality of $R$. This immediately implies that $R/(Q_i\cap R)\cong T/Q_i$, as rings, which means $Q_i\cap R$ is a prime ideal of $R$. Now it is clear that $(R:T)=\bigcap_{i\in I} (Q_i\cap R)$, and therefore $(R:T)$ is a semiprime ideal of $R$. The final part is evident by $(4)$ of Proposition \ref{t2}.
\end{proof}

\section{Finiteness conditions}
As mentioned in the introduction, if $R$ is a finite maximal subring of a ring $T$, then $T$ is finite too, see \cite{blgfc,blknfms,klein,laffey,lee}. If $T$ is a commutative ring, and $R$ is a maximal subring of $T$, then $R$ is an Artinian ring if and only if $T$ is an Artinian ring and $T$ is integral over $R$, see \cite[Theorem 3.8]{azkrc}. Moreover, if $R$ is a Noetherian maximal subring of a commutative ring $T$, then $T$ is Noetherian, by Hilbert Basis Theorem. Conversely, if $R$ is a maximal subring of a commutative Noetherian ring $T$ and $T$ is integral over $R$, then $R$ is a Noetherian ring by Eakin-Nagata-Eisenbud Theorem, see \cite[Theorem 3.98]{lam2}. In this section we obtain some observation for non-commutative rings by assuming some finiteness condition on maximal subrings. First we begin by the following result.

\begin{prop}\label{t20}
Let $R$ be a maximal subring of a ring $T$ and $P=(R:T)_l$. If $PT=T$, then $R=\mathbb{I}_T(P)$, $T$ is a finitely generated left $R$-module, $P$ is a finitely generated right $R$-module which is a right primitive ideal of $R$ and $T$ is not a left quasi duo ring. Moreover, if in addition $R$ is a left Artinian/Noetherian ring, then so is $T$.
\end{prop}
\begin{proof}
Since $P$ is a left ideal of $T$ which is a two sided ideal of $R$ and $P$ is not an ideal of $T$, we conclude that $R\subseteq \mathbb{I}_T(P)\subsetneq T$ and therefore by maximality of $R$ we deduce that $R=\mathbb{I}_T(P)$. From $PT=T$, we obtain that if $M$ is a maximal left ideal of $T$ which contains $P$, then $M$ is not an ideal of $T$ and therefore $T$ is not a left quasi duo ring. Since $T=PT$, we conclude that $1=y_1t_1+\cdots+y_nt_n$ for some $y_i\in P$, $t_i\in T$ and $n\in\mathbb{N}$. Thus for each $x\in T$ we have $x=x1=(xy_1)t_1+\cdots+(xy_n)t_n$. Now note that $y_i\in P$ which is a left ideal of $T$, therefore $xy_i\in P\subseteq R$ which immediately shows that $T=Rt_1+\cdots+Rt_n$. Also note that for each $p\in P$, we have $p=1p=y_1t_1p+\cdots+y_nt_np$. Now since $t_ip\in Tp\subseteq P\subseteq R$, we deduce that $P=y_1R+\cdots+y_nR$ and therefore $P$ is a finitely generated right $R$-module. Since ${}_RT$ is finitely generated, we deduce that $P$ is a right primitive ideal of $R$, by Corollary \ref{t5}. The final part is evident for ${}_RT$ is finitely generated.
\end{proof}

It is clear that the previous proposition has a similar statement for $(R:T)_r$. Therefore we have the following result.

\begin{cor}\label{t21}
Let $R$ be a maximal subring of a ring $T$ and $(R:T)\in Max(T)$. If $(R:T)\subsetneq (R:T)_l, (R:T)_r$ (i.e., $(R:T)_l$ and $(R:T)_r$ are not ideals of $T$), then the following hold:
\begin{enumerate}
\item ${}_R T$ and $T_R$ are finitely generated.
\item $(R:T)_l$ is a finitely generated right $R$-module which is a right primitive ideal of $R$. In particular, $l.ann_T((R:T)_l)=l.ann_T(p_1)\cap\cdots\cap l.ann_T(p_m)$ for some $p_1,\ldots p_m\in (R:T)_l$.
\item $(R:T)_r$ is a finitely generated left $R$-module which is a left primitive ideal of $R$. In particular, $r.ann_T((R:T)_r)=r.ann_T(q_1)\cap\cdots\cap r.ann_T(q_n)$ for some $q_1,\ldots q_n\in (R:T)_r$.
\end{enumerate}
In particular, if $(R:T)\in Max(T)\setminus Spec(R)$, then $(1)-(3)$ hold.
\end{cor}
\begin{proof}
Since $(R:T)\subsetneq (R:T)_l, (R:T)_r$, we conclude that $(R:T)_lT=T$ and $T(R:T)_r=T$, for $(R:T)$ is a maximal ideal of $T$. Hence we are done by the previous proposition and its proof. For the final part note that if $(R:T)$ is not a prime ideal of $R$, then by Lemma \ref{t1}, $(R:T)$ is not equal to either $(R:T)_l$ or $(R:T)_r$.
\end{proof}

In particular, if we add the assumption that $R$ is a left/right Noetherian (resp. Artinian) ring in the previous corollary, we have our next corollary.

\begin{cor}\label{t22}
Let $R$ be a maximal subring of a ring $T$ and $(R:T)\in Max(T)$. Then the following hold:
\begin{enumerate}
\item If $R$ is a left Noetherian (resp. Artinian) ring, then either $T$ is a left Noetherian (resp. Artinian) ring or $(R:T)=(R:T)_l$. Therefore $(R:T)_l\subseteq (R:T)_r$ (resp. $(R:T)=(R:T)_l=(R:T)_r$).
\item If $R$ is a right Noetherian (resp. Artinian) ring, then either $T$ is a right Noetherian (resp. Artinian) ring or $(R:T)=(R:T)_r$. Therefore $(R:T)_r\subseteq (R:T)_l$ (resp. $(R:T)=(R:T)_l=(R:T)_r$).
\end{enumerate}
\end{cor}
\begin{proof}
For $(1)$, note that either $(R:T)_l=(R:T)\subseteq (R:T)_r$ or $(R:T)\subsetneq (R:T)_l$ and therefore by Proposition \ref{t20},  ${}_R T$ is finitely generated and hence $T$ is a left Noetherian (resp. Artinian) $R$-module. Hence $T$ is a left Noetherian (resp. Artinian) ring. Also note that if $R$ is a left Artinian ring, then by Lemma \ref{t1}, $(R:T)_l$ and $(R:T)_r$ are maximal ideals of $R$, thus in the case $(R:T)_l=(R:T)\subseteq (R:T)_r$, we conclude that $(R:T)=(R:T)_l=(R:T)_r$.  The proof of $(2)$ is similar.
\end{proof}

Consequently, if $R$ is an integrally closed maximal subring of a zero-dimensional ring $T$ and $(R:T)$ is not a semiprime ideal of $R$, then by Theorem \ref{t34} and the previous corollary we deduce that $T$ is a left/right Noetherian (resp. Artinian) ring, whenever $R$ is a left/right Noetherian (resp. Artinian) ring. For the next result, note that since in a simple ring $T$, we automatically have the assumption $(R:T)=0$ for each maximal subring $R$ of $T$, then by the previous result we conclude the following.

\begin{cor}\label{t23}
Let $R$ be a maximal subring of a simple ring $T$. Then the following hold:
\begin{enumerate}
\item If $R$ is a left Noetherian (resp. Artinian) ring, then either $T$ is a left Noetherian (resp. Artinian) ring or $(R:T)_l=0$ (resp. $(R:T)_l=(R:T)_r=0$, $R=\mathbb{M}_n(D)$, where $D$ is a division ring and $T=\mathbb{M}_n(S)$, for a simple ring $S$ and $D$ is a maximal subring of $S$).
\item If $R$ is a right Noetherian (resp. Artinian) ring, then either $T$ is a right Noetherian (resp. Artinian) ring or $(R:T)_r=0$ (resp. $(R:T)_l=(R:T)_r=0$, $R=\mathbb{M}_n(D)$, where $D$ is a division ring and $T=\mathbb{M}_n(S)$, for a simple ring $S$ and $D$ is a maximal subring of $S$).
\end{enumerate}
\end{cor}
\begin{proof}
For $(1)$, the Noetherian part is clear by the previous result. Hence assume that $R$ is a left Artinian ring, if $T$ is not a left Artinian ring, then by the previous corollary $(R:T)_l=0$. Since $(R:T)_l$ is a prime ideal of $R$, we conclude that $R=\mathbb{M}_n(D)$, for some division ring $D$ and $n\geq 1$ (hence $R$ is a right Artinian too). Thus $T=\mathbb{M}_n(S)$, for a simple ring $S$, for $R$ is a subring of $T$ and $T$ is a simple ring. Since $R$ is a maximal subring of $T$, we immediately conclude that $D$ is a maximal subring of $S$. The proof of $(2)$ is similar.
\end{proof}

Similar to Proposition \ref{t19}, there exists a relation between the centralizer of a left Noetherian maximal subring $R$ of a ring $T$, the condition that $T$ is a left Noetherian ring and the conductor ideals, as follows.

\begin{prop}\label{t24}
Let $R$ be a left Noetherian maximal subring of a ring $T$. Then either $C_T(R)\subseteq R$ or $T$ is a left Noetherian ring and $(R:T)=(R:T)_l=(R:T)_r$.
\end{prop}
\begin{proof}
Assume that $C_T(R)$ is not contained in $R$ and $\alpha\in C_T(R)\setminus R$. Since $\alpha$ commutes with any element of $R$, we immediately conclude that $R[\alpha]$ is a left Noetherian ring, by Hilbert Basis Theorem. Clearly $T=R[\alpha]$, for $R$ is a maximal subring of $T$ and $\alpha\in T\setminus R$. Hence $T$ is a left Noetherian ring. Finally, if $(R:T)_l$ is not an ideal of $T$, then by Remark \ref{t13}, we deduce that $R=\mathbb{I}_T((R:T)_l)$ and therefore $R$ contains $C_T(R)$ which is absurd. Thus $(R:T)_l$ is an ideal of $T$ and similarly $(R:T)_r$ is an ideal of $T$. Hence $(R:T)=(R:T)_l=(R:T)_r$ and we are done.
\end{proof}

Now we have the following main result in this section.

\begin{thm}\label{t35}
Let $R$ be a Noetherian (resp. Artinian) maximal subring of a ring $T$ and $(R:T)\neq 0$. Then the following hold:
\begin{enumerate}
 \item $T/l.ann_T((R:T)_l)$, in particular, $T/l.ann_T((R:T))$, are left Noetherian (resp. Artinian) rings.
 \item $T/r.ann_T((R:T)_r)$, in particular $T/r.ann_T((R:T))$, are right Noetherian (resp. Artinian) rings.
 \item If $T$ is semiprime, then $T/ann_T((R:T))$ is a Noetherian (resp. an Artinian) ring.
 \item If $T$ is a prime ring, then $T$ is finitely generated as left and right $R$-modules. In particular, $T$ is a Noetherian (resp. an Artinian) ring. Moreover, $(R:T)_l$ and $(R:T)_r$ are  right and left primitive ideals of $R$, respectively.
\end{enumerate}
\end{thm}
\begin{proof}
$(1)$ Since $R$ is a right Noetherian ring, there exist $x_1,\ldots, x_n\in (R:T)_l$ such that $(R:T)_l=x_1R+\cdots+x_nR$. Thus $l.ann_T((R:T)_l)=l.ann_T(x_1)\cap\cdots\cap l.ann_T(x_n)$. Therefore $T/l.ann_T((R:T)_l)$ embeds in $Tx_1\times\cdots\times Tx_n$ as a left $T$-module and therefore as a left $R$-module too. Since $Tx_i\subseteq R$, we conclude that $T/l.ann_T((R:T)_l)$ embeds in $R^n$ as a left $R$-module and therefore $T/l.ann_T((R:T)_l)$ is left Noetherian (resp. Artinian) for $R$ is left Noetherian (resp. Artinian). Thus the first part of $(1)$ holds. Also note that $(R:T)\subseteq (R:T)_l$ and therefore $l.ann_T((R:T)_l)\subseteq l.ann_T((R:T))$. Hence $T/l.ann_T((R:T))$) is left Noetherian (resp. Artinian) too.  Similarly $(2)$ holds. For $(3)$, first note that since $T$ is a semiprime ring, we immediately conclude that the left and right annihilators of $(R:T)$ are coincided. Thus by $(1)$ and $(2)$, we deduce that $T/ann_T((R:T))$ is Noetherian (resp. Artinian). Finally, for $(4)$, since $0\neq (R:T)\subseteq (R:T)_l$, we infer that $l.ann_T((R:T)_l)=0$, thus by the proof of $(1)$, $T$ is isomorphic to a left submodule of $R^n$, hence $T$ is a finitely generated left $R$-module. Similarly, $T$ is a finitely generated right $R$-module. This immediately shows that $T$ is Noetherian (resp. Artinian). The final part is evident by Corollary \ref{t5}.
\end{proof}

Note that in the previous theorem, the assumption $(R:T)\neq 0$ is not necessary for items $(1)-(3)$, but since the conclusions in the case $(R:T)=0$ is trivial for these items, we assumed the condition $(R:T)\neq 0$ in the statement of the theorem.\\

In \cite{klein} and \cite{laffey}, the authors proved that if a finite ring $R$ is a maximal subring of a ring $T$, then $T$ is finite too. In the following remark we prove this result by the previous theorem in special case.

\begin{rem}\label{t39}
Let $T$ be a prime ring with a finite maximal subring $R$. If $(R:T)_l\neq 0$, then $T$ is finite. To see this, note that by the proof of $(4)$ in the previous theorem, $T$ embeds in $R^n$ and hence $T$ is finite.
\end{rem}

Also we have the following observations.

\begin{prop}\label{t39a}
Let $R$ be a left Noetherian (resp. Artinian) ring which is a maximal subring of a ring $T$. Assume that $(R:T)_l$ is a finitely generated as right ideal of $R$ and $(R:T)$ contains a prime ideal $Q$ of $T$. Then either $(R:T)=(R:T)_l\subseteq (R:T)_r$ (resp. $(R:T)=(R:T)_l=(R:T)_r$) or $T$ is a left Noetherian (resp. Artinian) ring.
\end{prop}
\begin{proof}
Similar to the proof of the previous theorem $T/l.ann_T((R:T)_l)$ is a left Noetherian (resp. Artinian) $R$-module. Since $l.ann_T((R:T)_l)(R:T)_l=0\subseteq Q$ and $Q$ is a prime ideal of $T$, we have two cases either $(R:T)_l\subseteq Q$ or $l.ann_T((R:T)_l)\subseteq Q$. If $(R:T)_l\subseteq Q$, then $(R:T)_l\subseteq (R:T)$ and therefore $(R:T)_l=(R:T)\subseteq (R:T)_r$ (resp. $(R:T)_l=(R:T)=(R:T)_r$ for $R$ is a left Artinian ring and $(R:T)_l$ and $(R:T)_r$ are prime ideals of $R$, by Lemma \ref{t1}). Hence assume that $l.ann_T((R:T)_l)\subseteq Q$ and therefore $l.ann_T((R:T)_l)\subseteq (R:T)\subseteq R$. Thus $l.ann_T((R:T)_l)$ is a left Noetherian (resp. Artinian) $R$-module. Therefore $T$ is a left Noetherian (resp. Artinian) $R$-module and hence is a left Noetherian (resp. Artinian) ring.
\end{proof}

\begin{prop}\label{t36}
Let $T$ be a ring which is not prime and $R$ be a Noetherian maximal subring of $T$ with $(R:T)\neq 0$. If $R$ contains a prime ideal $Q$ of $T$, then either $T$ is Noetherian or $Q=(R:T)=(R:T)_l$ or $Q=(R:T)=(R:T)_r$ is a minimal prime ideal of $T$ (thus $Q$ is unique). Moreover if $T$ is neither left Noetherian nor right Noetherian then $Q=(R:T)=(R:T)_l=(R:T)_r$.
\end{prop}
\begin{proof}
First note that $Q\subseteq (R:T)$. Now we have two cases: $(a)$ $Q\neq (R:T)$. This immediately implies that $l.ann_T((R:T)), r.ann_T((R:T))\subseteq Q$. Thus by $(1)$ and $(2)$ of the Theorem \ref{t35}, we conclude that $T/Q$ is Noetherian and since $Q\subseteq R$, we deduce that $T$ is Noetherian too. $(b)$ Hence assume that $Q=(R:T)$. Now suppose that $Q'\in Min(T)$ and $Q'\subsetneq Q$, thus $Q'\subsetneq (R:T)$ and therefore by the first case we conclude that $T$ is Noetherian and we are done. Hence $Q=(R:T)$ is a minimal prime ideal of $T$. Now assume that $(R:T)_l$ and $(R:T)_r$ are not contained in $Q=(R:T)$. Then we deduce that $l.ann_T((R:T)_l)$ and $r.ann_T((R:T)_r)$ are contained in $Q$, for $Q$ is a prime ideal. Therefore by $(1)$ and $(2)$ of the Theorem \ref{t35}, $T/Q$ is a left/right Noetherian ring and therefore $T$ is a Noetherian ring, for $Q\subseteq R$. By a similar argument and using $(1)$ and $(2)$ of the Theorem \ref{t35} and the fact that $Q$ is prime, we conclude the final part too.
\end{proof}

\begin{cor}\label{t37}
Let $R$ be an Artinian maximal subring of a prime ring $T$. Then either $R\cong \mathbb{M}_n(D)$ for a division ring $D$ (in particular, $T=\mathbb{M}_n(S)$, where $D$ is a maximal subring of $S$) or $T\cong\mathbb{M}_n(D')$ for a division ring $D'$.
\end{cor}
\begin{proof}
If $R$ is a prime ring, then clearly the first part of the statement of theorem holds. Hence assume that $R$ is not a prime ring. Hence $(R:T)_l\neq 0$ and therefore $l.ann_T((R:T)_l)=0$ for $T$ is prime. Thus by $(1)$ of Theorem \ref{t35}, we conclude that $T$ is a left Artinian ring and hence we are done (note $T$ is prime).
\end{proof}

Finally we conclude this paper by the following result.

\begin{prop}\label{t38}
Let $T$ be a ring which is not prime and $R$ be an Artinian maximal subring of $T$ with $(R:T)\neq 0$. Then either $T$ is Artinian or $dim(T)=0$ or $R$ contains a unique (minimal) prime ideal of $T$, say $Q$, and $T/Q\cong \mathbb{M}_n(S)$ where $S$ has a maximal subring $D$ which is a division ring.
\end{prop}
\begin{proof}
Assume that $T$ is not Artinian. We have two cases: $(a)$ each prime ideal $Q$ of $T$ is not contained in $R$. Thus $R+Q=T$, for $Q\neq 0$ (note $T$ is not prime) and $R$ is a maximal subring of $T$. Thus $R/(R\cap Q)\cong T/Q$, which immediately shows that $Q$ is a maximal ideal of $T$, i.e., $dim(T)=0$. $(b)$ Hence assume that there exists a prime $Q$ of $T$ such that $Q\subseteq R$. Therefore $Q\subseteq (R:T)\subseteq (R:T)_l\cap (R:T)_r$. Now, if $l.ann_T((R:T))$ is not contained in $Q$, we conclude that $(R:T)_l\subseteq Q\subseteq (R:T)\subseteq (R:T)_r$, for $Q$ is a prime ideal of $T$. Since $R$ is an Artinian ring and $(R:T)_l$, $(R:T)_r$ are prime ideals in $R$, we deduce that $(R:T)_l=Q=(R:T)=(R:T)_r$. Similarly if $r.ann_T((R:T)_r)$ is not contained in $Q$, the previous equalities hold. Thus assume that $Q$ contains $l.ann_T((R:T)_l)$ and $r.ann_T((R:T)_r)$. Hence by $(1)$ and $(2)$ or Theorem \ref{t35}, we deduce that $T/Q$ is an Artinian ring. Since $Q\subseteq R$, we conclude that $T$ is Artinian too which is absurd. Hence we have $Q=(R:T)=(R:T)_l=(R:T)_r$ which is a maximal ideal of $R$. Thus $R/Q\cong \mathbb{M}_n(D')$ for a division ring $D'$ and a natural number $n$. Since $R/Q$ is a maximal subring of $T/Q$, we conclude that $T/Q\cong\mathbb{M}_n(S)$, where $S$ has a maximal subring $D\cong D'$.
\end{proof}

\centerline{\Large{\bf Acknowledgement}}
The author is grateful to the Research Council of Shahid Chamran University of Ahvaz (Ahvaz-Iran) for
financial support (Grant Number: SCU.MM1402.721)


\end{document}